\documentclass[11pt]{amsart}
\usepackage{amsmath} 
\usepackage[normalem]{ulem}
\usepackage{amsthm}
\usepackage{amssymb} 
\usepackage{amsfonts}
\usepackage{graphicx} 
\usepackage{color,comment}
\usepackage{soul,xcolor}
\usepackage{ccaption}
\usepackage[shortlabels]{enumitem}
\usepackage{url}
\setlist[enumerate]{label=\normalfont{(\roman*)}}

\newtheorem{theorem}{Theorem}[section]
\newtheorem{lemma}[theorem]{Lemma}

\newtheorem{corollary}[theorem]{Corollary}

\theoremstyle{definition}

\newtheorem{definition}[theorem]{Definition}

\newtheorem{remark}[theorem]{Remark}

\setstcolor{red}

\let\int\relax
\newcommand{\int}{\mathring}

\usepackage[margin=1.5in]{geometry}

\title[Gluck twist and unknotting of satellite $2$-knots]{Gluck twist and unknotting of satellite $2$-knots}
\author[Seungwon Kim]{Seungwon Kim}
\begin{document}

\begin{abstract}
In this paper, we show that the Gluck twist of certain satellite $2$-knots in a $4$-manifold do not change the diffeomorphism type in three different ways: one is directly from the definition of the satellite $2$-knot, and the other two are by finding an equivalent description of the satellite $2$-knot. Furthermore, using the new description, we gave infinite number of new examples of $2$-knots which are unknotted by connected summing a single standard real projective plane. 
\end{abstract}

\maketitle

\section{Introduction}
Let $K$ be an embedded $2$-sphere in a $4$-manifold $X$ with a product neighborhood. Consider an operation which cuts a neighborhood $\nu(K)$ of $K$ and glues it back in a different way. By Gluck \cite{gluck1962embedding}, there are only two ways to glue it back, one is just the trivial gluing, and the other is called the \emph{Gluck twist}. Gluck \cite{gluck1962embedding} showed that the Gluck twist of $S^4$ is a homotopy $4$-sphere, which might give a potential counterexample to the smooth $4$-dimensional Poincar\'e conjecture. (See Kirby's problem 4.23 \cite{kirbyproblems}.)
 
Many knotted spheres (i.e., $2$-knots) in $S^4$ are known to be Gluck twist trivial, such as ribbon $2$-knots \cite{gluck1962embedding,yanagawa1969ribbon}, twist spun knots \cite{gordon1976foursphere, paononlinear}, certain union of two ribbon disks \cite{NashStipsicz}, twist roll spun knots \cite{naylor2020gluck}, tube sums of all such $2$-knots \cite{habiro2000gluck}, and more generally, $2$-knots 0-concordant to all such knots \cite{Melvin, Sunukjianconcordanceisotopysurgery}. 

%Melvin\cite{Melvin}, showed that the Gluck twist problem in $S^4$ is equivalent to the isotopy problem in $\mathbb{C}P^2$. More precisely, the Gluck twist along a $2$-knot $K$ in $S^4$ is standard if and only if the pair $(\mathbb{C}P^2, K \# \mathbb{C}P^1)$ obtained by blowing up a point in $K$ is pairwise homotopic to the standard pair $(\mathbb{C}P^2, \mathbb{C}P^1)$. A surface in $\mathbb{C}P^2$ is called a \emph{unit} surface if it is obtained by blowing up a knotted surface in $S^4$. From the Melvin's result, it is obvious that if a unit sphere is isotopic to the standard $\mathbb{C}P^2$, then the Gluck twist along the corresponding surface in $S^4$ is trivial. Gabai\cite{gabai2020} asked whether every unit sphere is isotopic to the standard $\mathbb{C}P^1$ or not.

In this paper, we consider the Gluck twist problem of a \emph{satellite $2$-knot}, which is defined below:

\begin{definition}
Let $P$ and $C$ be $2$-knots embedded in $S^4$ and a $4$-manifold $X$ respectively. Assume $C$ has a product neighborhood in $X$. Consider a simple loop $\gamma \subset S^4-\nu(P)$. Then there exists a diffeomorphism $\rho : \overline{S^4 - \nu(\gamma)} \rightarrow \nu(C)$, where $\nu(\cdot)$ denotes a neighborhood of $\cdot$ in a $4$-manifold. Let $K = \rho(P) \subset X$. We call $K$ the \emph{satellite $2$-knot in $X$ of companion $C$ with pattern $(P,V)$}. Equivalently, 
$$(X, K) = ((\overline{X - \nu(C))}\bigcup_{\partial \rho}(\overline{S^4 - \nu(\gamma)}), P),$$ where
$$\partial \rho = \rho\restriction_{\partial(\overline{S^4 - \nu(\gamma)})}\; : \partial(\overline{S^4 - \nu(\gamma)}) \rightarrow \partial \nu(C) \simeq \partial (\overline{X - \nu(C)})$$
 and 
$$P \subset S^4 - \nu(\gamma) \subset (\overline{X - \nu(C))}\bigcup_{\partial \rho}(\overline{S^4 - \nu(\gamma)}) \simeq X.$$
 We say a satellite $2$-knot is degree $n$ if $[P] = n \in H_2(S^4 - \nu(\gamma)) \simeq \mathbb{Z}$.
Especially, we call a satellite $2$-knot is a \emph{$cable$ $2$-knot} if its pattern $2$-knot $P$ is the unknot.
\end{definition}

In \cite{hughes2020isotopies}, Hughes, Miller and the author studied the Gluck twist of a satellite $2$-knot and showed that when the pattern is 0-concordant to a tube sum of twist spun knots and the degree in $H_2(\overline{S^4 - \nu(\gamma)})\simeq H_2(S^2 \times D^2) \simeq \mathbb{Z}$ is zero, then the Gluck twist along the satellite $2$-knot is trivial. Furthermore, if the degree is one, then the Gluck twist of the satellite $2$-knot is the same as the Gluck twist along its companion. 
  
In this paper, we extend the above result to a more general setting:

\begin{theorem}\label{main}
Let $K$ be a satellite $2$-knot in a $4$-manifold $X$ of companion $C$ with pattern $(P,V)$. Then the following holds:
\begin{enumerate}
    \item If the degree of $K$ is even, then the Gluck twist of $X$ along $K$ is diffeomorphic to the Gluck twist of $X$ along $P \subset D^4 \subset X$ .
    %(resp. $K$ is isotopic to $(\#^n C) \# \mathbb{C}P^1$ in $X \# \mathbb{C}P^2$ after blow up.)
    \item If the degree of $K$ is odd then the Gluck twist of $X$ along $K$ is diffeomorphic to the Gluck twist of $X$ along $C \# P$.
    %(resp. $K$ is isotopic to the standard $\mathbb{C}P^1$ in $X \# \mathbb{C}P^2$ after blow up). 
\end{enumerate}
\end{theorem}

%\begin{theorem}\label{main}
%Let $K$ be a satellite $2$-knot in a $4$-manifold $X$ of companion $C$ with pattern $P$. Suppose that the Gluck twist of $S^4$ along $P$ is trivial. Then the following holds:
%\begin{enumerate}
%    \item If $[P] \in H_2(S^4 - \nu(\gamma)) \simeq \mathbb{Z}$ is even, then the Gluck twist of $X$ along $K$ is diffeomorphic to $X$.
%    %(resp. $K$ is isotopic to $(\#^n C) \# \mathbb{C}P^1$ in $X \# \mathbb{C}P^2$ after blow up.)
%    \item If $[P] \in H_2(S^4 - \nu(\gamma))\simeq \mathbb{Z}$ is odd then the Gluck twist of $X$ along $K$ is diffeomorphic to the Gluck twist of $X$ along $C$.
%    %(resp. $K$ is isotopic to the standard $\mathbb{C}P^1$ in $X \# \mathbb{C}P^2$ after blow up). 
%\end{enumerate}
%\end{theorem}

Theorem \ref{main} has the following immediate corollary:

\begin{theorem}\label{corollary}
Let $K$ be a satellite $2$-knot in a $4$-manifold $X$ of companion $C$ with pattern $(P,V)$. Suppose that the Gluck twist of $S^4$ along $P$ is trivial. Then the following holds:
\begin{enumerate}
    \item If the degree of $K$ is even, then the Gluck twist of $X$ along $K$ is diffeomorphic to $X$.
    %(resp. $K$ is isotopic to $(\#^n C) \# \mathbb{C}P^1$ in $X \# \mathbb{C}P^2$ after blow up.)
    \item If the degree of $K$ is odd and the Gluck twist of $X$ along $C$ is trivial, then the Gluck twist of $X$ along $K$ is diffeomorphic to $X$.
    %(resp. $K$ is isotopic to the standard $\mathbb{C}P^1$ in $X \# \mathbb{C}P^2$ after blow up). 
\end{enumerate}
\end{theorem}

Hence, for example, the Gluck twist of $S^4$ along a satellite $2$-knot of a twist spun $2$-knot companion with a twist spun $2$-knot pattern is trivial.
%
%\begin{theorem}\label{corollary}
%Let $K$, $P$, $C$, $X$ be as in Theorem \ref{main}. Suppose that $X \simeq S^4$. Then the following holds:
%\begin{enumerate}
%    \item If $[P]$ is even, then the Gluck twist along $K$ is diffeomorphic to $S^4$.
%    %(resp. $K$ is isotopic to $(\#^n C) \# \mathbb{C}P^1$ in $X \# \mathbb{C}P^2$ after blow up.)
%    \item If $[P]$ is odd and the Gluck twist of $S^4$ along $C$ is trivial, then the Gluck twist of $S^4$ along $K$ is diffeomorphic to $S^4$.
%    %(resp. $K$ is isotopic to the standard $\mathbb{C}P^1$ in $X \# \mathbb{C}P^2$ after blow up). 
%\end{enumerate}
%\end{theorem}

We give three different proofs. The first proof is directly from the definition of the satellite $2$-knot, and the other two are by finding an equivalent description of the satellite $2$-knot. Especially, third proof is done by unknotting cable $2$-knots using a standard real projective plane.

Viro \cite{viro1973local} found the first example of a $2$-knot in $S^4$ which can be unknotted by connected summing a single standard real projective plane. As far as the author is aware, there were no previously known example of non-ribbon $2$-knots which can be unknotted by connected summing with a single standard real projective plane.

In this paper, We extend the Viro's result to non-ribbon $2$-knots.

\begin{theorem}\label{extensionviro}
There exist an infinite number of non-ribbon $2$-knots which are unknotted by connected summing a single standard real projective plane.
\end{theorem}

%
%\begin{theorem}\label{nonorientable}
%Let $\overline{C}$ be a cable $2$-knot in a $4$-manifold $X$ of companion $C$ with unknotted pattern $(U,W)$. Let $R$ be a standard projective plane in $D^4 \subset X$. Then, if the degree of $\overline{C}$ is even, then $\overline{C}\# R$ is isotopic to $R$ and if the degree is odd, then $\overline{C} \#R$ is isotopic to $C\#R$.
%\end{theorem}
%
%
%\begin{theorem}\label{twistspunnotribbon}
%If a companion is a $n$-twist spun $2$-knot $\tau_n (k)$ of a $1$-knot $k$, then its cable is not a ribbon.
%\end{theorem}

%We show that every even cable $2$-knot is unknotted by connected summing single standard real projective plane. Furthermore, we showed that there are infinite cable $2$-knots which are not ribbon.

%\item 
%\item cable unknot, P
%\item satellite, blowup, cable, connected sum, Twist

%\begin{theorem}
%Let $K \subset S^4$ be a satellite $2$-knot with companion $C$ and a pattern $P \subset (S^4 - \nu(\gamma)) \simeq S^2 \times D^2 \subset S^4$. Suppose that the Gluck twist of $S^4$ along $P$ is trivial(resp. $P$ is trivial in $\mathbb{C}P^2$ after blow up). Then the following holds:
%\begin{enumerate}
%    \item If $[P] \in H_2(S^4 - \nu(\gamma)) \simeq \mathbb{Z}$ is even, then the Gluck twist along $K$ is trivial(resp. $K$ is trivial in $\mathbb{C}P^2$ after blow up)
%    \item If $[P] \in H_2(S^4 - \nu(\gamma))\simeq \mathbb{Z}$ is odd and the the Gluck twist along $C$ is trivial, then the Gluck twist along $K$ is is trivial(resp. $K$ is trivial in $\mathbb{C}P^2$ after blow up). 
%\end{enumerate}
%\end{theorem}
\subsection*{Acknowledgements}
The author would like to thank Hongtaek Jung, Maggie Miller, Jason Joseph and Hannah Schwartz for many helpful conversations about the earlier draft. The author was supported by the Institute for Basic Science (IBS-R003-D1) at the time of this project.

\section{Proofs of Main Theorem}

\begin{proof}[First proof of Theorem \ref{main}]
Consider a banded unlink diagram of $P$. $\gamma$ can be isotoped so that it can be seen as the unknot in the banded unlink diagram of $P$. Note that the obvious disk bounded by $\gamma$ intersects the banded unlink diagram in $n$ times where $n \equiv [P]\pmod{2}$. Without loss of generality, we also can assume that the obvious disk does not intersect the bands.

We can think of the Gluck twist in the following way: First, get the natural handle decomposition from a banded unlink diagram of a $2$-knot in a $4$-manifold. Then, we add the $+1$-framed circle to a meridian of one of the dotted circles. Then, this Kirby diagram represents the Gluck twist of the given $4$-manifold along the $2$-knot with the given banded unlink diagram. See Figure \ref{examplebud} for an example.

\begin{figure}
\centering
\includegraphics[width=\textwidth]{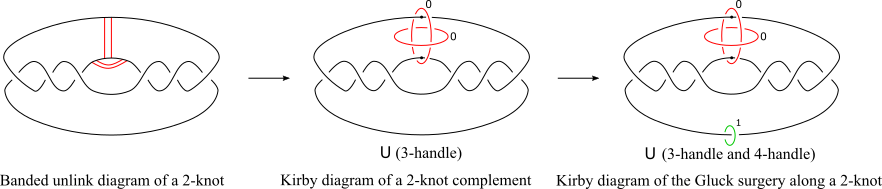}
\caption{First figure is a banded unlink diagram of a spun trefoil in $S^4$. From this banded unlink, we can get the natural handle decomposition of its complement, by putting a dot on each unlink, and changing each band to a $0$-framed circle as in the second figure. If we put a $+1$-framed circle to a meridian of one of the dotted circles, then we get a Kirby diagram of the Gluck twist of $S^4$ along the spun trefoil knot.}
\label{examplebud}
\end{figure}

Also, the Gluck twist along $K$ can be thought as follows: We first do the Gluck twist $\overline{S^4 - \nu(\gamma)}$ along $P$ and glue it to $X - \nu(C)$ without twist. See the below equations.\\

$$(\overline{X - \nu(K)}) \bigcup_{\phi} (S^2 \times D^2) = \Big(\overline{\big((\overline{X - \nu(C)}) \bigcup_{\partial \rho} (\overline{S^4 - \nu(\gamma)}) - \nu(P)}\Big) \bigcup_{\phi} (S^2 \times D^2)$$

$$=\big((\overline{X - \nu(C)}) \bigcup_{\partial \rho} (\overline{S^4 - \nu(\gamma) - \nu(P)})\big) \bigcup_{\phi} (S^2 \times D^2)$$

$$=(\overline{X - \nu(C)}) \bigcup_{\partial \rho}\big( (\overline{S^4 - \nu(\gamma) - \nu(P)}) \bigcup_{\phi} (S^2 \times D^2)\big).$$

Here, $\phi$ is the self-diffeomorphism of $S^2 \times S^1$ which gives the Gluck twist. Note that $\partial \rho$ is the identity map of $S^2 \times S^1$.

Consider an isotopy of $\gamma$ through the $+1$-framed circle. This isotopy will link $\gamma$ and the $+1$-framed circle. However, we can push $\gamma$ down along the gradient flow of the Morse function so that $\gamma$ is sitting inside the $1$-handle body, do an isotopy of $\gamma$ in the $1$-handlebody which does not touch the dotted circles, and push it back to the level of the Kirby diagram so that $\gamma$ is unlinked from $+1$-framed circle. See Figure \ref{isotopy} for the actual moves in a Kirby diagram, and Figure \ref{kirby} for an example. The $+1$-framed circle can be moved to any meridian of the dotted circles since it is isotopic to a meridian of $P$. Therefore, we can keep doing it so that $\gamma$ is unlinked from the every dotted circle and every attaching circle of the $2$-handles.

\begin{figure}
    \centering
    \includegraphics[width=\textwidth]{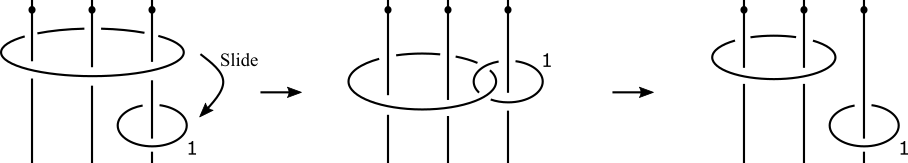}
    \caption{We slide the curve $\gamma$ through the $+1$-framed circle to get the figure in the middle from the figure in the left. Then, we push $\gamma$ below the index-$2$ critical points to sit on the $1$-handlebody. Then, we can do small isotopy which does not intersect the dotted circles. Then, we push $\gamma$ back to the original level to get the figure in the right.}
    \label{isotopy}
\end{figure}

\begin{figure}
\centering
 \includegraphics[width=\textwidth]{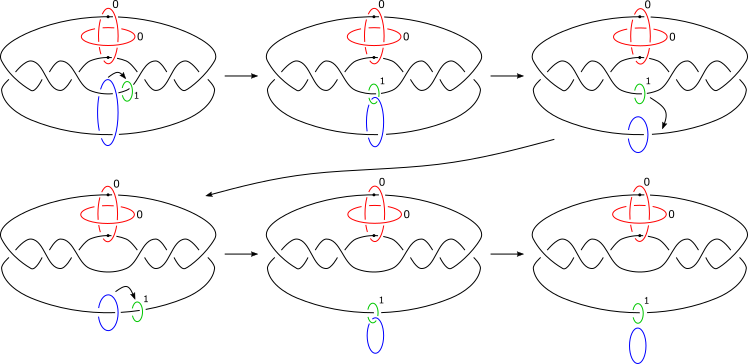}
 \caption{In the first figure, we slide the blue curve $\gamma$ along the $+1$-framed circle to get the second figure. In the second figure, we push $\gamma$ down to the level below the index-$2$ critical points, do small isotopy in that level, and push it back to the original level to get the third figure. We can always move the $+1$-framed circle to a meridian of any dotted circles, so we can move it like the fourth figure. Then we do the same moves to get the last figure.}   
\label{kirby}
\end{figure}

We can figure out the result of the Gluck twist if we can specify the framing of $\gamma$. Each time we push $\gamma$ through the $+1$-framed circle, framing changes by $1$. All the other isotopies such as pushing up and down along the gradient flow, and the isotopy in the level below the index-$2$ critical points which does not touch the dotted circles do not change the framing. Hence, the final framing of $\gamma$ differs from the original framing by the number of times that it passes through the $+1$-framed circle mod 2, since $\gamma$'s framing is either $0$ or $1$. This number is same as the parity of the degree of $K$. Hence, the Gluck twist of $X$ along $K$ is only depend on the degree of $K$. Hence, the Gluck twist of $X$ along an even degree satellite $2$-knot is diffeomorphic to the Gluck twist of $X$ along a degree zero satellite $2$-knot, which is just the pattern $2$-knot. Also, the Gluck twist of $X$ along an odd degree satellite $2$-knot is diffeomorphic to the Gluck twits of $X$ along a degree one satellite $2$-knot, which is just a connected sum of the companion and the pattern by \cite{kanenobu1983groups}. This proves our main theorem.

%Hence, when the degree of $K$ is even, the framing is $0$, so the gluing is trivial, i.e., we get $X$. Otherwise, the framing is $1$, and the resultant manifold is diffeomorphic to the Gluck twist of $X$ along $C$. This completes the proof of our main theorem.

\end{proof}

\begin{definition}
Let $J$ and $K$ be disjoint two $2$-knots in a $4$-manifold $X$. Consider a $3$-ball $B = D^2 \times I$ embedded in $X - (J \cup K)$ such that $J \cap B = D^2 \times \{0\}$ and $K \cap B = D^2 \times \{1\}$. Consider a $2$-knot $L = (J-D^2 \times \{0\}) \cup (\partial D^2 \times I) \cup (K - D^2 \times \{1\}$. We call $L$ a \emph{tube sum of $J$ and $K$} and $L\cap B = \partial D^2 \times I$ a \emph{tube} connecting $J$ and $K$. We call the procedure to make $L$ from $J$ and $K$ a \emph{tubing} $J$ and $K$. See Figure \ref{tubeexample} for a banded unlink diagram of a tube.  
\end{definition}

\begin{figure}
\centering
\includegraphics[width=0.5\textwidth]{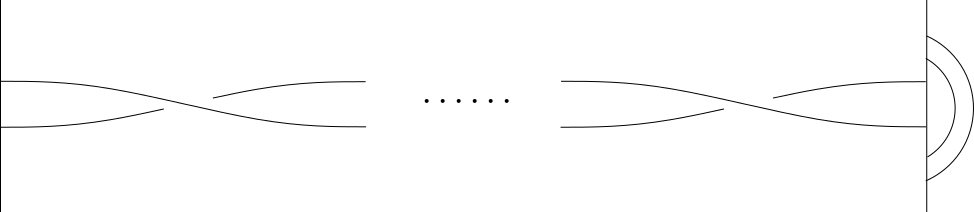}
\caption{A banded unlink diagram of a tube.}
\label{tubeexample}
\end{figure}

\begin{remark}
The tube sum is sometimes called the \emph{band sum} (e.g., \cite{hughes2020isotopies}), since it is a higher dimensional analogue of the band sum of $1$-knots. It is a generalization of the connected sum.
\end{remark}

\begin{definition}
Let $C$ be a $2$-knot in a $4$-manifold $X$ with $\nu(C) \simeq S^2 \times D^2$. Assume $C = S^2 \times \{0\}$. \emph{Parallel copies of $C$} are disjoint union of $2$-spheres $\bigcup_i S^2 \times \{p_i\} \subset \nu(C) \subset X$ (see Figure \ref{parallelcopies} for a banded unlink diagram in $S^2 \times D^2$, represented by an exterior of $\gamma \in S^4$). 
\end{definition}

\begin{figure}
\centering
\includegraphics[width=0.3\textwidth]{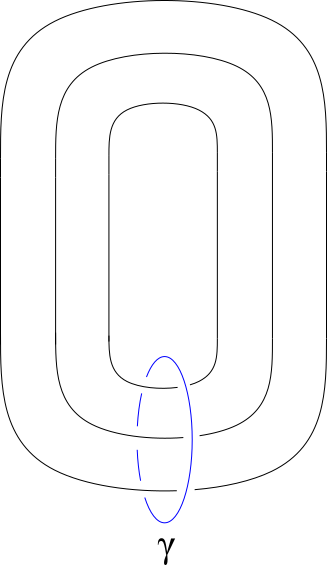}
\caption{A banded unlink diagram of parallel copies of a $2$-knot $C$. This figure is drawn in $S^2 \times D^2 \simeq \nu(C)$, represented by an exterior of $\gamma \in S^4$.}
\label{parallelcopies}
\end{figure}

\begin{definition}
Let $\bigcup_i C_i$, $i= 1, \cdots, n$ be a set of parallel copies of a $2$-knot $C$ in a $4$-manifold $X$. Without loss of generality, assume that $D^2$ is a unit disk in $\mathbb{C}^2$, $C_i = S^2 \times \{p_i\}$, where $p_i \in [-1, 1] \subset D^2$, and if $i< j$ then $p_i < p_j$. Let $\delta = \{q_j\} \times [p_j, p_{j+1}] \subset S^2 \times [p_j, p_{j+1}]$ be an arc which connects $C_j$ and $C_{j+1}$. Consider an $\epsilon$-neighborhood of $\delta$ in $X$ and a tube in that neighborhood connecting $C_j$ and $C_{j+1}$ which its core is $\delta$. We call a tube is \emph{untwisted}, if a tube is $\partial D_{\epsilon} \times [p_j, p_{j+1}]$, where $D_{\epsilon} \subset S^2$ is a small disc neighborhood of $q_j$. Otherwise, we call a tube is \emph{half-twisted}. See Figure \ref{tubes} for their banded unlink diagrams.
\end{definition}

\begin{figure}
\centering
\includegraphics[width = 0.6\textwidth]{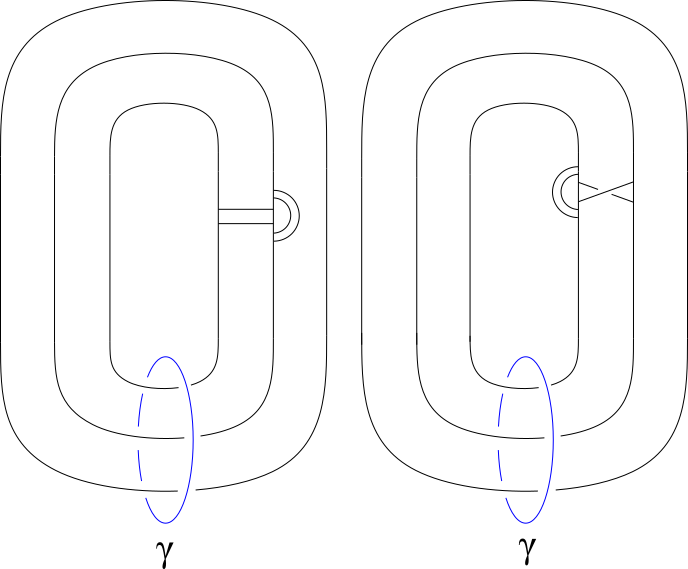}
\caption{Left: A banded unlink diagram of untiwsted tube. Right: A banded unlink diagram of half-twisted tube.}
\label{tubes}
\end{figure}

\begin{remark}
A banded unlink diagram of a tube can be twisted many times as in Figure \ref{tubeexample}, but we always can cancel full twist by local isotopies of bands. Therefore, there is no ambiguity in the definition of half-twisted.
\end{remark}

\begin{lemma}\label{satellitelemma}
Every satellite $2$-knot $K$ is a tube sums of a pattern $2$-knot $P \subset D^4 \subset X$ and parallel copies of a companion $2$-knot $C$.
\end{lemma}

\begin{proof}
We only need to show that the pattern $(P,V)$ is a tube sum of a pattern $2$-knot $P \subset D^4 \subset V$ and parallel copies of the core $2$-sphere of $V$. Then these parallel copies of the core $2$-sphere becomes parallel copies of $C$ after identifying $V$ to $\nu(C)$, and $P$ becomes a $2$-knot in $D^4 \subset X$, hence, we proved the lemma. We draw a banded unlink diagram of $(P,V)$ as the right most figure of Figure \ref{newsatellite}. Then we can see that the surface represented by the right most figure is isotopic to the surface represented by the left most figure of Figure \ref{newsatellite}, which is a tube sum of $P \subset D^4 \subset V$ and parallel copies of the core of $V$.
\end{proof}

\begin{figure}
\centering
\includegraphics[width=\textwidth]{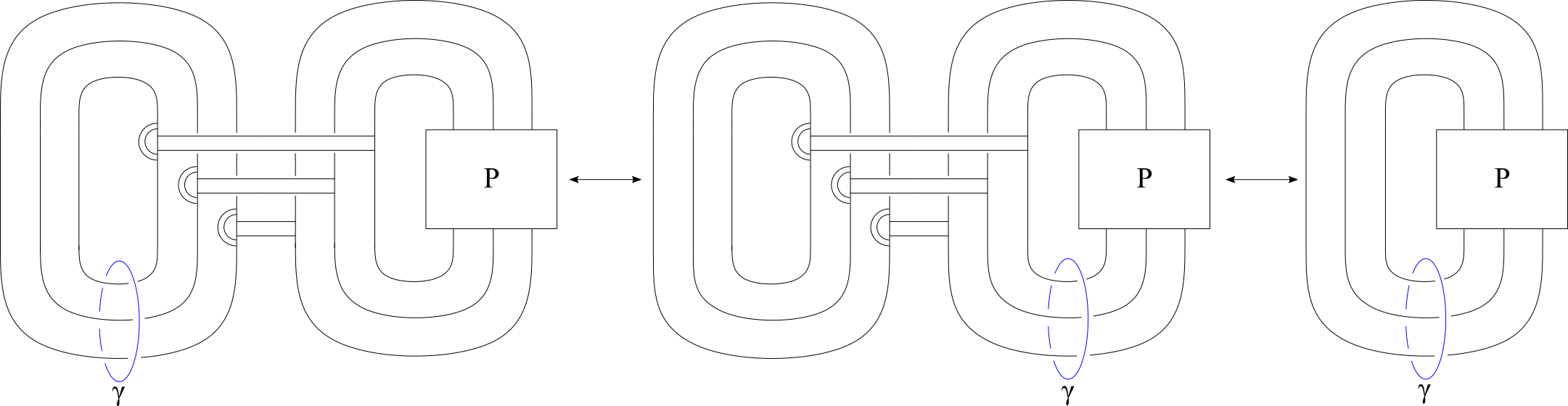}
\caption{Equivalent banded unlink diagrams of $(P,V)$. From the right most figure, we can stabilize sufficient number of times to get the middle figure, and then isotope $\gamma$ as in the right most figure.}
\label{newsatellite}
\end{figure}

\begin{remark}
This lemma is true for $n$-knot with $n \geq 1$, with suitable generalizations of the notions, such as the tube sum and parallel copies. 
\end{remark}
\smallskip
\begin{proof}[Second proof of Theorem \ref{main}]
In \cite{habiro2000gluck, Melvin}, it is shown that if a $2$-knot is obtained by tubing a $2$-link, then the Gluck twist along a $2$-knot is same as Gluck twist along a $2$-link.  
Hence, by \cite{habiro2000gluck, Melvin} and Lemma \ref{satellitelemma}, the Gluck twist of $X$ along $K$ is same as the Gluck twist of $X$ along a $2$-link, which is a split union of parallel copies of $C$ and $P \subset D^4 \subset X$. Consider a pair of $2$-spheres in parallel copies of $C$. We can tube them with a untwisted tube and isotope resulting $2$-sphere to a unknotted $2$-sphere as in Figure \ref{unknotting}. (There is another way to understand this isotopy. These two $2$-spheres cobound a $S^2 \times I$. Tubing these two with untwisted tubes will remove a $3$-ball $D^2 \times I$ from $S^2 \times I$, which makes the resulting $2$-sphere bounds a $3$-ball, so it is unknotted $2$-sphere in a $4$-ball.)

If $[P]$ is even, then we can pair them all to get a split union of unknotted $2$-spheres, and if $[P]$ is odd, then we can pair all $2$-spheres except one to get a split union of unknotted $2$-spheres and $C$. Then by \cite{habiro2000gluck, Melvin} again, the Gluck twist along a parallel copies of $C$ is is trivial if $[P]$ is even, and is same as the Gluck twist along $C$ if $[P]$ is odd. Therefore, the Gluck twist of $X$ along $K$ is diffeomorphic to the Gluck twist of $X$ along $P$ if the degree of $K$ is even, and the Gluck twist of $X$ along $K$ is diffeomorphic to the Gluck twist of $X$ along the split union of $C$ and $P$ if the degree of $K$ is odd. By \cite{habiro2000gluck, Melvin} again, the Gluck twist of $X$ along the split union of $C$ and $P$ is diffeomorphic to the Gluck twist of $X$ along $C \# P$.

\begin{figure}
\centering
\includegraphics[width=\textwidth]{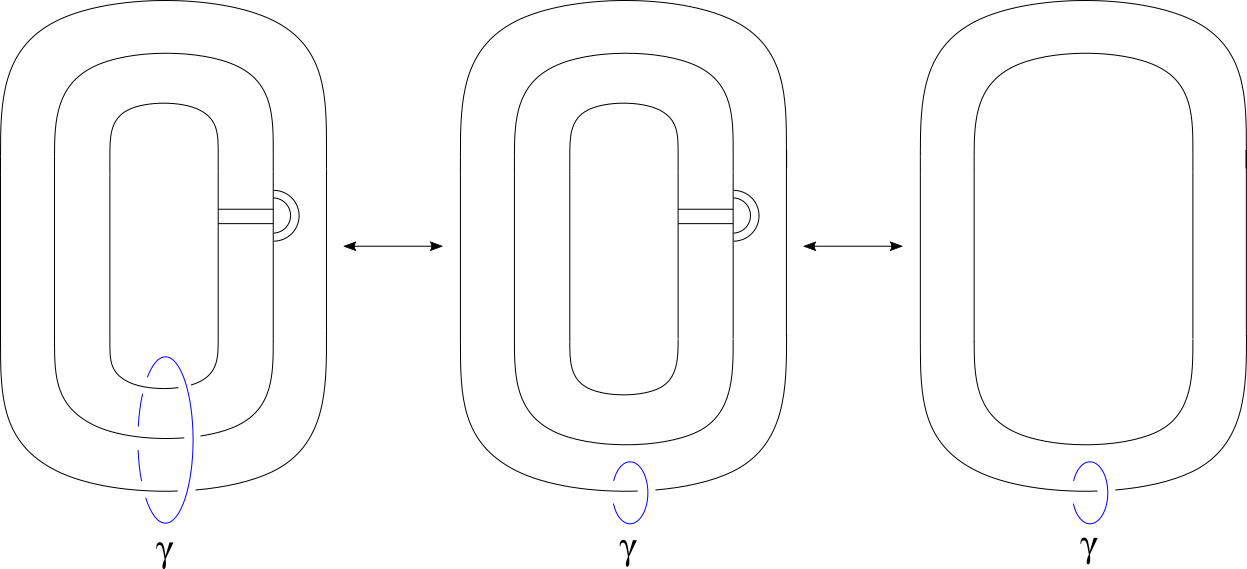}
\caption{We can pair adjacent $2$-spheres using a untwisted tube. Then we can isotope $\gamma$ through the tube, and isotope tubed pair to one unknotted $2$-sphere in a ball.}
\label{unknotting}
\end{figure}      
\end{proof}

\begin{lemma}\label{cableshape}
A cable $2$-knot can be obtained from parallel copies of a companion by connecting adjacent copies with half-twisted tubes. 
\end{lemma}

\begin{proof}
We only need to show that the pattern can be obtained from parallel copies of unknotted 2-spheres by connecting adjacent copies with half-twisted tubes. Since $\pi_1 (S^4 - U) = \mathbb{Z}$, where $U$ is a unknot in $S^4$, the pattern of the cable is only determined by the $[\gamma] \in \pi_1 (S^4 - U) = \mathbb{Z}$. Let $[\gamma] = n, n \in \mathbb{Z}$. Consider a banded unlink diagram as in Figure \ref{cable}, which is obtained by taking $n$-parallel copies of unknotted 2-spheres and connecting them with half-twisted tubes. Note that without $\gamma$ this banded unlink diagram is just a different banded unlink diagram of an unknotted 2-sphere $U$. Then $\gamma$ links the banded unlink diagram $n$-times, which implies that $[\gamma] = n$. Hence, the satellite of the pattern represented by this banded unlink diagram is just a cable $2$-knot with degree $n$.   
\end{proof}

%\begin{figure}
%\includegraphics[width=0.4 \textwidth]{halftwistedtube}
%\caption{Half-twisted tube.}
%\label{halftwistedtube}
%\end{figure}

\begin{figure}
\includegraphics[width=0.3\textwidth]{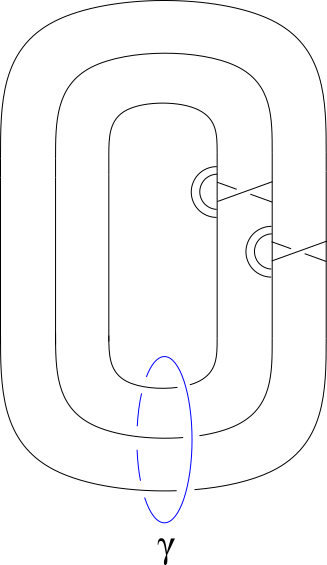}
\caption{A banded unlink diagram of a unknot pattern with $[\gamma] = 3 \in \mathbb{Z}$.}
\label{cable}
\end{figure}

%The following theorem is Theorem 0.1 in \cite{katanaga1999gluck}.
%\begin{theorem}\cite{katanaga1999gluck}\label{price}
%The Gluck twist of a $4$-manifold $X$ along a $2$-knot $K$ is diffeomorphic to the Price twist of $X$ along a $K \# R$ where $R \subset D^4 \subset X$ is a standard projective plane in $D^4$. 
%\end{theorem}
%
%The following theorem is an immediate consequence of the above theorem.
%
%\begin{corollary}\cite{katanaga1999gluck}\label{pricecor}
%If a $2$-knot is unknotted by connected summing a standard projective plane, then its Gluck twist is trivial.
%\end{corollary}
%
\begin{theorem}\label{nonorientable}
Let $K$ be a cable $2$-knot in a $4$-manifold $X$ of companion $C$ with unknotted pattern $(P,V)$. Let $R$ be a standard projective plane in $D^4 \subset X$. Then, if the degree of $K$ is even, then $K\# R$ is isotopic to $R$ and if the degree is odd, then $K\#R$ is isotopic to $C\#R$.
\end{theorem}

\begin{proof}
The effect of connected summing $R$ in $K$ is just adding half-twisted band to a banded unlink diagram of $K$. We can put this half-twisted band near a half-twisted tube of a cable $2$-knot. Then by Figure \ref{untwisting}, we always can untwist the tube using $\mathbb{R}P^2$. Then by the similar argument as in the second proof of Theorem \ref{main}, each adjacent pair of $2$-spheres connected by a tube is isotoped to an unknotted $2$-sphere in a small ball in $X$. Hence, if the degree is even, then $K \# R$ can be isotoped to $R$, and if the degree is odd, then $K\#R$ is isotoped to $C\#R$. 
\end{proof}

\begin{figure}
\centering
\includegraphics[width=\textwidth]{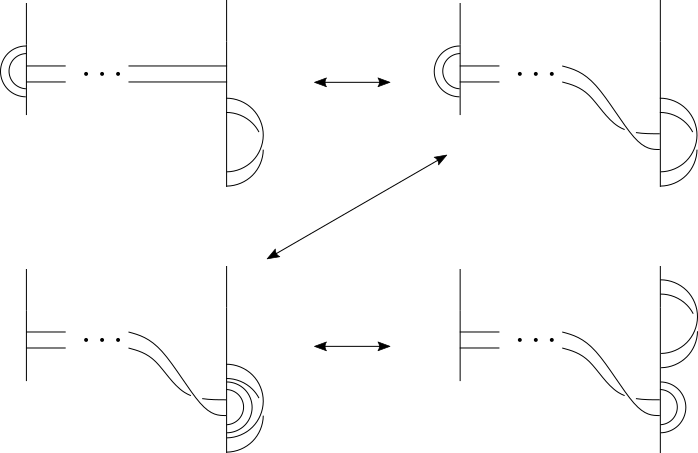}
\caption{Half-twisted tubes can be untwisted by a half-twisted band.}
\label{untwisting}
\end{figure}

Note that if a companion is ribbon $2$-knot, then its cable is ribbon. It is natural to expect that a cable is not ribbon if its companion is not ribbon. 

\begin{theorem}\label{twistspunnotribbon}
If a companion is a $n$-twist spun $2$-knot $\tau_n (k)$ of a $1$-knot $k$ with $|n| \geq 2$, then its cable is not a ribbon.
\end{theorem}

\begin{proof}
By Zeeman \cite{zeeman1965twisting}, any $\tau_n (k)$ is fibered by a once-punctured $n$-fold cyclic branched covering of $S^3$ over $k$. 

By Kanenobu \cite{kanenobu1983groups}, the degree $m$ cable of $\tau_n (k)$ is fibered by a once-punctured $m$-fold connect sum of the $n$-fold cyclic branched cover of $S^3$ over $k$.
%
%By Kanenobu \cite{kanenobu1983groups}, the fiber of a degree $m$ cable of $\tau_n (k)$ is a once-punctured $3$-manifold, which the $3$-manifold is obtained by connected summing $m$-copies of $n$-fold cyclic branched coverings of $S^3$ over $k$. 

To be a fibered ribbon $2$-knot, Cochran \cite{cochran1983ribbon} showed that its fiber should be a once punctured $\#_q S^1 \times S^2$, however, Plotnick \cite{plotnick1984finite} showed that a cyclic branched coverings of $S^3$ over a non-trivial knot does not admit $S^1 \times S^2$ summands. Hence, a cable of a $\tau_n(k)$ is not a ribbon $2$-knot. 
\end{proof}

%\begin{corollary}
%There exists an infinite number of non-ribbon $2$-knots which are unknotted by connected summing a single standard real projective plane.
%\end{corollary}

\begin{proof}[Proof of Theorem \ref{extensionviro}]
Consider $\tau_2(k)$ where $k$ is a $2$-bridge knot, which its $2$-fold branched covering is a lens space $L(p,q)$. Then $2n$-cable of $\tau_2 (k)$ is fibered by $\#_{2n} L(p,q) - D^3$. Then $2n_i$-cable and $2n_j$-cable of $\tau_2 (k)$ are not isotopic if $n_i \neq n_j$ since the infinite cyclic cover of a fibered $2$-knot has the homotopy type of the fiber. Since by Theorem \ref{nonorientable} and \ref{twistspunnotribbon}, $2n$-cables of $\tau_2 (k)$ are non-ribbon $2$-knots which are unknotted by connected summing a single standard real projective plane, the set of $2n$-cables of $\tau_2 (k)$ is a desired set of $2$-knots.

%Consider $\tau_n(k)$ and its $2n$-cables. Then if $n_i \neq n_j$, then $2n_i$-cable is not isotopic to $2n_j$ since the homotopy types are different. Also, by Theorem \ref{nonorientable} and \ref{twistspunnotribbon}, $2n$-cables of $\tau_2 (k)$ are non-ribbon $2$-knots which are unknotted by connected summing a single standard real projective plane. Hence the set of $2n$-cables of $\tau_n(k)$ is a desired set of $2$-knots.
\end{proof}

The following theorem is Theorem 0.1 in \cite{katanaga1999gluck}.
\begin{theorem}\cite{katanaga1999gluck}\label{price}
The Gluck twist of a $4$-manifold $X$ along a $2$-knot $K$ is diffeomorphic to the Price twist of $X$ along a $K \# R$ where $R \subset D^4 \subset X$ is a standard projective plane in $D^4$. 
\end{theorem}

The following theorem is an immediate consequence of the above theorem.

\begin{corollary}\cite{katanaga1999gluck}\label{pricecor}
If a $2$-knot is unknotted by connected summing a standard projective plane, then its Gluck twist is trivial.
\end{corollary}

\smallskip
\begin{proof}[Third Proof of Theorem \ref{main}]
Before going into the proof, we first define some terminology.

\begin{definition}
A $2$-knot $K'$ in $X \# \mathbb{C}P^2$ is called \emph{unit $2$-knot} if a pair $(K',X \# \mathbb{C}P^2)$ is obtained from a pairwise connected sum of the standard pair $(\mathbb{C}P^1, \mathbb{C}P^2)$ and a $2$-knot pair $(K, X)$. 
\end{definition}

\begin{lemma}\label{blowupisotopy}
Let $K'$ be unit $2$-knot obtained from a satellite $2$-knot $K$ with companion $C$ with pattern $(P,V)$. Then $K'$ is isotopic to a unit $2$-knot obtained from a connected sum of a pattern $2$-knot and a cable $2$-knot $\overline{C}$ with companion $C$.   
\end{lemma}
\begin{proof}
Let $(U, W)$ be the pattern of $\overline{C}$. It is suffices to show that $(\mathbb{C}P^1, \mathbb{C}P^2)$ $\# (P,V)$ (see left figure of Figure \ref{isotopecable} for its banded unlink diagram) is isotopic to $(\mathbb{C}P^1, \mathbb{C}P^2) \# (U,W) \# (P, S^4)$ (see right figure of Figure \ref{isotopecable}). We need to show that these two figures represent isotopic surfaces. If we achieve this isotopy, we can cancel two $2$-spheres which are tubed by an untwisted tube, hence, cobound a $3$-ball, so that each adjacent $2$-spheres are tubed by half-twisted tubes. Then by Lemma \ref{cableshape}, we prove the lemma. 

To acheive the isotopy, we first do band slides so that only one tube is connecting $P$ to the rest of the surface. During the isotopy, tubes can be tangled with the diagram. Then as in \ref{HKM}, we can untangle tubes from $P$. After we untangle all such tubes, we move $+1$-framed circle to the left, and do further isotopy as before, to arrange every tube as we desired. 
 
\end{proof}
\begin{figure}
\centering
\includegraphics[width=\textwidth]{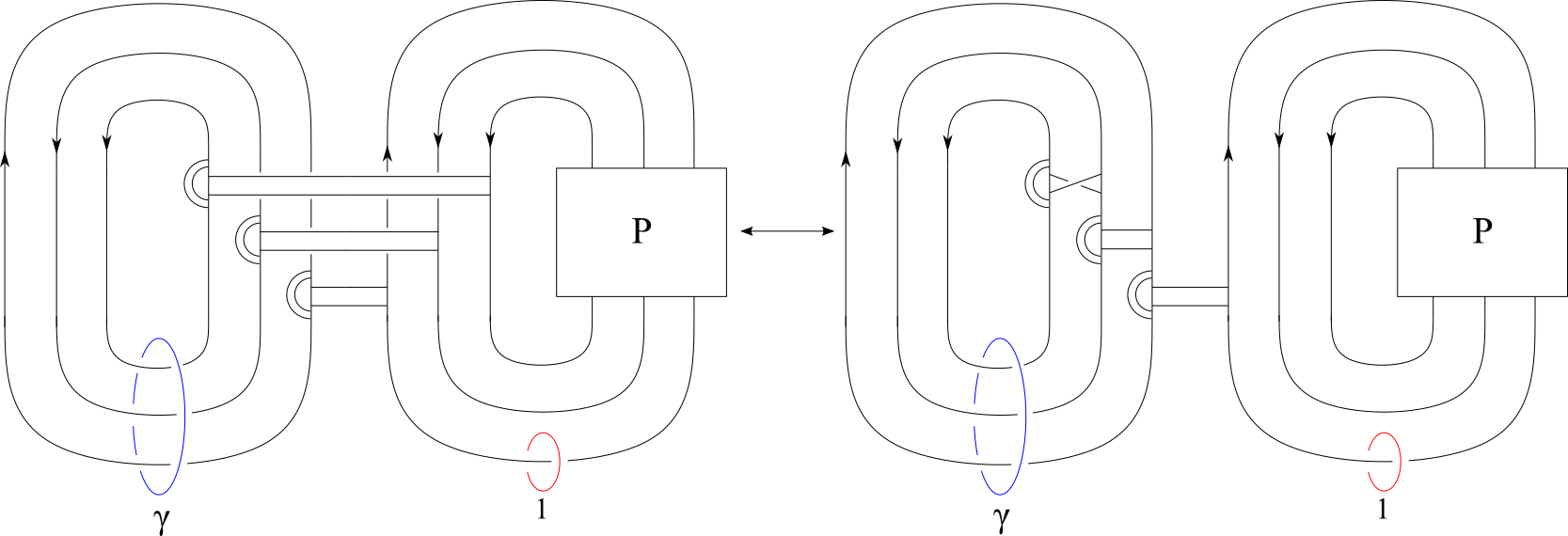}
\caption{Banded unlink diagrams of $(\mathbb{C}P^1, \mathbb{C}P^2) \# (P,V)$(left) and $(\mathbb{C}P^1, \mathbb{C}P^2) \# (U,W) \# (P, S^4)$(right). Note two parallel $2$-spheres connected by a untwisted tube can be isotoped to a unknotted $2$-sphere as in Figure \ref{unknotting} and cancelled. }
\label{isotopecable}
\end{figure}

\begin{figure}
\center
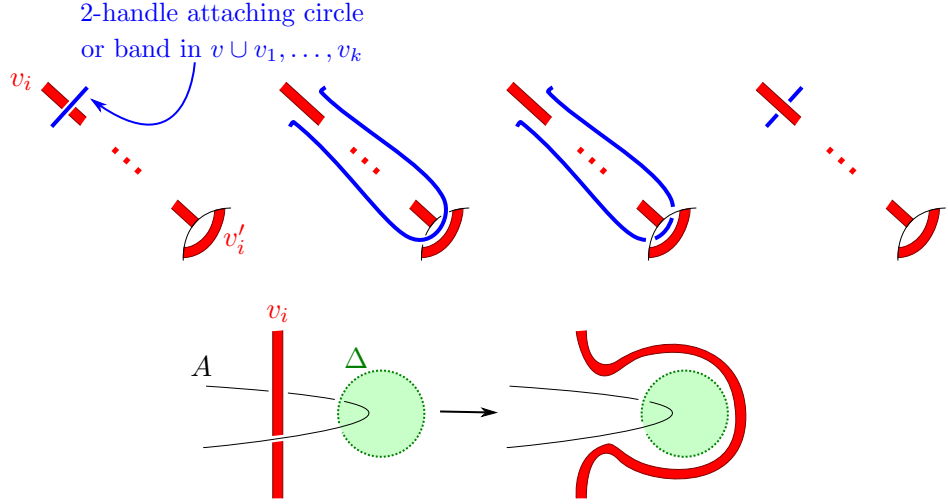
\label{HKM}
\caption{The figure from \cite{hughes2020isotopies}. The pair $v_i$ and $v'_i$ in the figure represents the tube. We always can change crossings between bands and tubes/2-handles. Also, if we have geometric dual sphere (in our case, $\mathbb{C}P^1$ represented by $+1$-framed circle), then we can change crossings between unlinks and tubes.} 
\end{figure}

In \cite{Melvin}, Melvin showed that the two $2$-knots have the same Gluck twist if unit $2$-knots obtained from them are isotopic. Then by Lemma \ref{blowupisotopy}, the Gluck twist of $X$ along $K$ is diffeomorphic to the Gluck twist of the connected sum of $\overline{C}$ and $P \subset D^4 \subset X$.
By Theorem \ref{price}, the Gluck twist of $X$ along $\overline{C} \# P$ is diffeomorphic to the Price twist of $X$ along $\overline{C} \# P \# R$, where $R$ is a standard projective plane in $D^4 \subset X$. Then by Theorem \ref{nonorientable}, if the degree of $K$ is even, $K \# R$ is isotopic to $P \# R$, and if the degree is odd, then $K \# R$ is isotopic to $C \# P \# R$. Hence, the Gluck twist of $X$ along $K$ is diffeomorphic to the Gluck twist of $X$ along $P$ if the degree is even, otherwise, it is diffeomorphic to the Gluck twist of $X$ along $C \# P$.
\end{proof}

\bibliographystyle{amsplain}
\bibliography{ref}
\end{document}